\documentclass{article}
\usepackage{amssymb,amsmath,amsthm}
\usepackage{epsfig}
\newcommand{\ie}{\emph{i.e.}}

\newcommand{\cf}{\emph{cf}}

\newcommand{\Real}{\mathbb{R}}

\newcommand{\sii}{L^2}
\newcommand{\sobi}{\mathop{W_0^{1,2}}\nolimits}

\newcommand{\eps}{\varepsilon}
\newcommand{\Hilbert}{\mathcal{H}}
\numberwithin{equation}{section}
\theoremstyle{plain}
\newtheorem{Theorem}{Theorem}[section]
\newtheorem{Lemma}[Theorem]{Lemma}
\theoremstyle{remark}
\newtheorem{Remark}[Theorem]{Remark}
\begin{document}
%
%
\title{\Large\textbf{%
Spectrum of the Laplacian in a narrow curved strip
with combined Dirichlet and Neumann boundary conditions%
}}
\author{David Krej\v{c}i\v{r}\'{\i}k}
\date{
\footnotesize
\begin{center}
\emph{
Department of Theoretical Physics, Nuclear Physics Institute,
\\
Academy of Sciences, 250\,68 \v{R}e\v{z} near Prague, Czech Republic
\smallskip \\
\emph{E-mail:} krejcirik@ujf.cas.cz
}
\end{center}
6 March 2008}
\maketitle
\begin{abstract}
\noindent
We consider the Laplacian in a domain squeezed
between two parallel curves in the plane,
subject to Dirichlet boundary conditions on one of the curves
and Neumann boundary conditions on the other.
We derive two-term asymptotics for eigenvalues
in the limit when the distance between the curves tends to zero.
The asymptotics are uniform and local in the sense that
the coefficients depend only on the extremal points where
the ratio of the curvature radii of the Neumann boundary
to the Dirichlet one is the biggest.
We also show that the asymptotics can be obtained
from a form of norm-resolvent convergence
which takes into account the width-dependence
of the domain of definition of the operators involved.
\bigskip
\begin{itemize}
\item[\textbf{MSC\,2000:}]
35P15; 49R50; 58J50; 81Q15.
\item[\textbf{Keywords:}]
Laplacian in tubes;
Dirichlet and Neumann boundary conditions;
dimension reduction; norm-resolvent convergence;
binding effect of curvature;
waveguides.
\bigskip
\item[\textbf{To appear in:}]
ESAIM: Control, Optimisation and Calculus of Variations 
\\
\verb|http://www.esaim-cocv.org|
\end{itemize}
\end{abstract}
%
%
\newpage
%
\section{Introduction}
%
Given an open interval $I\subseteq\Real$ (bounded or unbounded),
let $\gamma\in C^2(\overline{I};\Real^2)$
be a unit-speed planar curve.
The derivative~$\dot{\gamma}\equiv(\dot{\gamma}^1,\dot{\gamma}^2)$
and $n:=(-\dot{\gamma}^2,\dot{\gamma}^1)$
define unit tangent and normal vector fields along~$\gamma$, respectively.
The curvature is defined through the Frenet-Serret formulae by
$\kappa:=\det(\dot{\gamma},\ddot{\gamma})$;
it is a bounded and uniformly continuous function on~$I$.

For any positive~$\eps$, we introduce a mapping~$\mathcal{L}_\eps$
from $\overline{I} \times [0,1]$ to~$\Real^2$ by
\begin{equation}\label{StripMap}
  \mathcal{L}_\eps(s,t) := \gamma(s) + \eps \, t \, n(s) \,.
\end{equation}
Assuming that~$\mathcal{L}_\eps$ is injective
and that~$\eps$ is so small that the supremum norm of~$\kappa$
is less than~$\eps^{-1}$,
$\mathcal{L}_\eps$~induces a diffeomorphism
and the image
\begin{equation}\label{strip}
  \Omega_\eps :=
  \mathcal{L}_\eps\big(I \times (0,1)\big)
\end{equation}
has a geometrical meaning
of an open non-self-intersecting strip,
contained between the parallel curves~$\gamma(I)$ and
$
  \gamma_\eps(I) := \mathcal{L}_\eps(I\times\{1\})
$,
and, if~$\partial I$ is not empty, the straight lines
$\mathcal{L}_\eps\big(\{\inf I\}\times(0,1)\big)$
and $\mathcal{L}_\eps\big(\{\sup I\}\times(0,1)\big)$.
The geometry is set in such a way that $\kappa>0$
implies that the parallel curve~$\gamma_\eps$
is ``locally shorter'' than~$\gamma$, and vice versa,
\cf~Figure~\ref{figure}.

Let~$-\Delta_{DN}^{\Omega_\eps}$ be the Laplacian in $\sii(\Omega_\eps)$
with Dirichlet and Neumann boundary conditions
on~$\gamma$ and~$\gamma_\eps$, respectively.
If~$\partial I$ is not empty, we impose Dirichlet
boundary conditions on the remaining parts of~$\partial\Omega_\eps$.

For any self-adjoint operator~$H$ which is bounded from below,
we denote by $\{\lambda_j(H)\}_{j=1}^\infty$
the non-decreasing sequence of numbers corresponding to
the spectral problem of~$H$ according to
the Rayleigh-Ritz variational formula~\cite[Sec.~4.5]{Davies}.
Each~$\lambda_j(H)$ represents either a (discrete) eigenvalue
(repeated according to multiplicity) below the essential spectrum
or the threshold of the essential spectrum of~$H$.
All the eigenvalues below the essential spectrum
may be characterized by this variational/minimax principle.

Under the above assumptions,
our main result reads as follows:
\begin{Theorem}\label{Thm.mine}
For all $j \geq 1$,
\begin{equation}\label{expansion}
  \lambda_j(-\Delta_{DN}^{\Omega_\eps})
  \ = \ \left(\frac{\pi}{2\eps}\right)^2
  + \frac{\inf\kappa}{\eps}
  + o(\eps^{-1})
  \qquad\mbox{as}\qquad
  \eps \to 0
  \,.
\end{equation}
\end{Theorem}
\begin{figure}[t]
\begin{center}
\epsfig{file=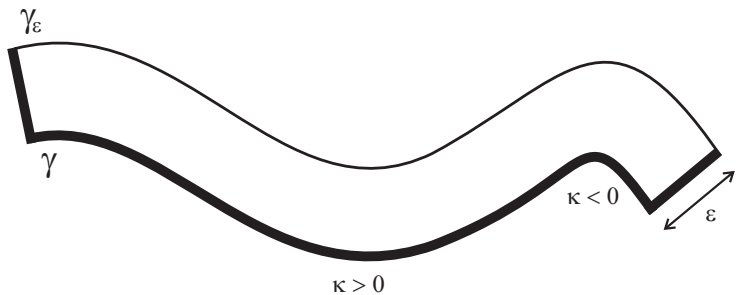,width=0.8\textwidth}
\end{center}
\caption{
The geometry of the strip $\Omega_\eps$ in a case of bounded~$I$.
The thick and thin lines correspond to Dirichlet
and Neumann boundary conditions, respectively.
}
\label{figure}
\end{figure}

Theorem~\ref{Thm.mine} has important consequences
for the spectral properties of the operator $-\Delta_{DN}^{\Omega_\eps}$,
especially in the physically interesting situation $I=\Real$.
In this case, assuming that the curvature~$\kappa$ vanishes at infinity,
the leading term $\pi^2/(2\eps)^2$ of~\eqref{expansion}
coincides with the threshold of the essential spectrum of
$-\Delta_{DN}^{\Omega_\eps}$.
The next term in the expansion then tells us that
\begin{enumerate}
\item[(a)]
the discrete spectrum exists
whenever~$\kappa$ assumes a negative value
and~$\eps$ is sufficiently small;
\item[(b)]
the number of the eigenvalues increases to infinity
as $\eps \to 0$.
\end{enumerate}
This provides an insight into the mechanism
which is behind the qualitative results
obtained by Dittrich and K\v{r}{\'\i}\v{z}
in their 2002 letter~\cite{DKriz2}.
Using $-\Delta_{DN}^{\Omega_\eps}$ as a model
for the Hamiltonian of a quantum waveguide,
they show that the discrete eigenvalues exist if, and only if,
the reference curve~$\gamma$ of sign-definite~$\kappa$
is curved ``in the right direction'',
namely if the Neumann boundary condition is imposed
on the ``locally longer'' boundary
(\ie~$\kappa<0$ in our setting),
and that~(b) holds.
The results were
further generalized in \cite{KKriz,FK3},
numerically tested in~\cite{OM},
and established in a different physical model in~\cite{JLP}.

The purely Dirichlet or Neumann strips
differ from the present situation in many respects
(see~\cite{KKriz} for a comparison).
The case of the Neumann Laplacian~$-\Delta_{N}^{\Omega_\eps}$
is trivial in the sense that
$$
  \lambda_1(-\Delta_{N}^{\Omega_\eps})
  = \lambda_1(-\Delta_N^I)
  = 0
  \,,
$$
independently of the geometry and smallness of~$\eps$,
where $-\Delta_{N}^I$ denotes the Neumann Laplacian in $\sii(I)$.
For $j \geq 2$, one has
$$
  \lambda_j(-\Delta_{N}^{\Omega_\eps})
  = \lambda_j(-\Delta_N^I) + o(1)
  \qquad\mbox{as}\qquad
  \eps \to 0
  \,,
$$
independently of the geometry.
More generally, it is well known that the spectrum
of the Neumann Laplacian on an $\eps$-tubular neighbourhood
of a Riemannian manifold converges when $\eps \to 0$
to the spectrum of the Laplace-Beltrami operator
on the manifold~\cite{Schatzman_1996}.

As for the Dirichlet Laplacian~$-\Delta_{D}^{\Omega_\eps}$,
it is well known~\cite{ES,GJ,DE,KKriz}
that the existence of discrete spectrum
in unbounded strips is robust,
\ie\ independent of the sign of~$\kappa$.
This is also reflected in the asymptotic formula, $j \geq 1$,
\begin{equation}\label{Dirichlet}
  \lambda_j(-\Delta_{D}^{\Omega_\eps})
  = \left(\frac{\pi}{\eps}\right)^2
  + \lambda_j\big(-\Delta_D^I - \frac{\kappa^2}{4}\big)
  + o(1)
  \qquad\mbox{as}\qquad
  \eps \to 0
  \,,
\end{equation}
known for many years \cite{Karp-Pinsky_1988,DE},
where $-\Delta_{D}^I$ denotes the Dirichlet Laplacian in $\sii(I)$.
That is, contrary to Theorem~\ref{Thm.mine},
in the purely Dirichlet case
the second term in the asymptotic expansion is independent of~$\eps$,
always negative unless~$\gamma$ is a straight line,
and its value is determined by the \emph{global} geometry of~$\gamma$.

The local character of~\eqref{expansion} rather resembles
the problem of a straight narrow strip of \emph{variable} width
studied recently by Friedlander and Solomyak
\cite{Friedlander-Solomyak_2007a,Friedlander-Solomyak_2007b},
and also by Borisov and Freitas~\cite{BF}
-- see also~\cite{F} for related work.
In view of their asymptotics,
the spectrum of the Dirichlet Laplacian
is basically determined by the points
where the strip is the widest.
In our model the cross-section is uniform
but the curvature \emph{and} boundary conditions
are not homogeneous.

The purely Dirichlet case with uniform cross-section
differs from the present situation
also in the direct method of the proof of~\eqref{Dirichlet}.
Using the parametrization~\eqref{StripMap},
the spectral problem for the Laplacian $-\Delta_{D}^{\Omega_\eps}$
in the ``curved'' and $\eps$-dependent Hilbert space $\sii(\Omega_\eps)$
is transferred to the spectral problem
for a more complicated operator~$H_\eps^D$
in $\sii\big(I\times(0,1)\big)$.
Inspecting the dependence of the coefficients of~$H_\eps^D$ on~$\eps$,
it turns out that the operator is in the limit $\eps \to 0$
decoupled into a sum of the ``transverse'' Laplacian
multiplied by~$\eps^{-2}$
and of the $\eps$-independent Schr\"odinger operator on~$\gamma$.
At this stage, the minimax principle is sufficient
to establish~\eqref{Dirichlet}.
Furthermore, since the ``straightened'' Hilbert space is independent of~$\eps$,
it is also possible to show that~\eqref{Dirichlet} is
obtained as a consequence of
some sort of norm-resolvent convergence \cite{DE,FK4}.
An alternative approach is based on the $\Gamma$-convergence method~\cite{BMT}.
See also~\cite{Grieser} for a recent survey of 
the thin-limit problem in a wider context.

The above procedure does not work in the present situation
because the transformed operator~$H_\eps^{DN}$
does not decouple as $\eps \to 0$,
at least at the stage of the elementary usage of the minimax principle.
Moreover, the operator domain of~$H_\eps^{DN}$
becomes dependent on~$\eps$;
contrary to the Dirichlet boundary condition,
the Neumann one is transferred
to an $\eps$-dependent and variable Robin-type boundary condition
(\cf~Remark~\ref{Rem.domain} below).
In this paper we propose an alternative approach,
which enables us to treat the case of combined boundary conditions.
Our method of proof is based on refined applications
of the minimax principle.

In the following Section~\ref{Sec.proof},
we prove Theorem~\ref{Thm.mine} as a consequence
of upper and lower bounds to $\lambda_j(-\Delta_{DN}^{\Omega_\eps})$.
More specifically, these estimates imply
\begin{Theorem}\label{Thm.stronger}
For all $j \geq 1$,
\begin{equation}\label{stronger}
  \lambda_j(-\Delta_{DN}^{\Omega_\eps})
  = \left(\frac{\pi}{2\eps}\right)^2
  + \lambda_j\big(-\Delta_D^I + \frac{\kappa}{\eps}\big)
  + \mathcal{O}(1)
  \qquad\mbox{as}\qquad
  \eps \to 0
  \,.
\end{equation}
\end{Theorem}
\noindent
Then Theorem~\ref{Thm.mine} follows at once as
a weaker version of Theorem~\ref{Thm.stronger},
by using known results about the strong-coupling/semiclassical asymptotics
of eigenvalues of the one-dimensional Schr\"odinger operator.
Indeed, for all $j \geq 1$, one has
\begin{equation}\label{strong}
  \lambda_j\big(-\Delta_D^I + \frac{\kappa}{\eps}\big)
  = \frac{\inf\kappa}{\eps} + o(\eps^{-1})
  \qquad\mbox{as}\qquad
  \eps \to 0
  \,.
\end{equation}
This result seems to be well known;
we refer to~\cite[App.~A]{FK1} for a proof in any dimension.

Another goal of the present paper is to show
that the eigenvalue convergence of Theorem~\ref{Thm.mine}
can be obtained as a consequence of the norm-resolvent ``convergence''
of~$-\Delta_{DN}^{\Omega_\eps}$ to $-\Delta_D^I+\kappa/\eps$
as $\eps \to 0$.
We use the quotation marks because the latter operator
is in fact $\eps$-dependent and the norm-resolvent convergence
should be rather interpreted as the convergence
of the difference of corresponding resolvent operators in norm.
However, the operators act in different Hilbert spaces
and the norm-resolvent convergence still requires
a meaningful reinterpretation.
Because of the technical complexity,
we postpone the statement of this convergence result
until Section~\ref{Sec.norm}.

The paper is concluded by Section~\ref{Sec.end}
in which we discuss possible extensions
of our main results.

\section{Spectral convergence}\label{Sec.proof}
%
In this section we give a simple proof of Theorem~\ref{Thm.stronger}
by establishing upper and lower bounds to $\lambda_j(-\Delta_{DN}^{\Omega_\eps})$.
We begin with necessary geometric preliminaries.

\subsection{Curvilinear coordinates}\label{Sec.coord}
%
As usual, the Laplacian $-\Delta_{DN}^{\Omega_\eps}$
is introduced as the self-adjoint operator in $\sii(\Omega_\eps)$
associated with the quadratic form~$Q_{DN}^{\Omega_\eps}$
defined by
\begin{align*}
  Q_{DN}^{\Omega_\eps}[\Psi]
  &:= \int_{\Omega_\eps} |\nabla\Psi(x)|^2 \, dx
  \,,
  \\
  \Psi \in D(Q_{DN}^{\Omega_\eps}) &:=
  \left\{
  \Psi \in W^{1,2}(\Omega_\eps) \ | \quad
  \Psi = 0 \quad \mbox{on} \quad
  \partial\Omega_\eps \setminus \gamma_\eps(I)
  \right\}
  \,.
\end{align*}
Here $\Psi$ on~$\partial\Omega_\eps$
is understood in the sense of traces.
It is natural to express the Laplacian
in the ``coordinates'' $(s,t)$ determined
by the inverse of~$\mathcal{L}_\eps$.

As stated in Introduction,
under the hypotheses that~$\mathcal{L}_\eps$ is injective and
\begin{equation}\label{Ass.basic}
  \eps \,\sup|\kappa| < 1 \,,
\end{equation}
the mapping~\eqref{StripMap} induces a global diffeomorphism
between $I\times(0,1)$ and~$\Omega_\eps$.
This is readily seen by the inverse function theorem
and the expression for the Jacobian
$
  \det(\partial_1\mathcal{L}_\eps,\partial_2\mathcal{L}_\eps)
  =\eps h_\eps
$
of~$\mathcal{L}_\eps$, where
\begin{equation}\label{Jacobian}
  h_\eps(s,t) := 1 - \kappa(s) \, \eps \, t
  \,.
\end{equation}
In fact, \eqref{Ass.basic}~yields the uniform estimates
\begin{equation}\label{uniform}
  0 <
  1 - \eps \sup\kappa
  \leq h_\eps \leq
  1 - \eps \inf\kappa
  < \infty
  \,,
\end{equation}
where the lower bound ensures that
the Jacobian never vanishes in $\overline{I}\times[0,1]$.

The passage to the natural coordinates
(together with a simple scaling)
is then performed via the unitary transformation
\begin{equation*}
  U_\eps : \sii(\Omega_\eps) \to
  \Hilbert_\eps := \sii\big(I\times(0,1),h_\eps(s,t)\,ds\,dt\big):
  \left\{\Psi\mapsto \sqrt{\eps} \ \Psi\circ\mathcal{L}_\eps\right\}
  \,.
\end{equation*}
This leads to a unitarily equivalent operator
$H_\eps:=U_\eps(-\Delta_{DN}^{\Omega_\eps})U_\eps^{-1}$ in~$\Hilbert_\eps$,
which is associated with the quadratic form~$Q_\eps$ defined by
\begin{align*}
  Q_\eps[\psi] &:=
  \int_{I\times(0,1)} \frac{|\partial_1\psi(s,t)|^2}{h_\eps(s,t)} \, ds\,dt
  + \int_{I\times(0,1)} \frac{|\partial_2\psi(s,t)|^2}{\eps^2}
  \, h_\eps(s,t) \, ds\,dt
  \,,
  \\
  \psi \in D(Q_\eps) &:=
  \left\{
  \psi \in W^{1,2}\big(I\times(0,1)\big) \ | \
  \psi = 0 \quad \mbox{on} \quad
  \partial\big(I\times(0,1)\big) \setminus \big(I\times\{1\}\big)
  \right\}
  .
\end{align*}

As a consequence of~\eqref{uniform},
$\Hilbert_\eps$ and $\sii\big(I\times(0,1)\big)$
can be identified as vector spaces
due to the equivalence of norms,
denoted respectively by $\|\cdot\|_\eps$ and $\|\cdot\|$
in the following.
More precisely, we have
\begin{equation}\label{norms}
  1 - \eps \sup\kappa
  \leq \frac{\|\psi\|_{\eps}^2}{\|\psi\|^2} \leq
  1 - \eps \inf\kappa
  \,.
\end{equation}
That is, the fraction of norms
behaves as $1+\mathcal{O}(\eps)$ as $\eps \to 0$.

\subsection{Upper bound}
%
Let~$\psi$ be a test function from the domain $D(Q_\eps)$
of the form
\begin{equation}\label{chi}
  \psi(s,t) := \varphi(s) \chi_1(t) \,,
  \qquad\mbox{where}\qquad
  \chi_1(t):=\sqrt{2} \sin\left(\pi t/2\right)
\end{equation}
and~$\varphi \in \sobi(I)$ is arbitrary.
Note that~$\chi_1$ is a normalized eigenfunction
corresponding to the lowest eigenvalue of $-\Delta_{DN}^{(0,1)}$,
\ie\ the Laplacian in $\sii((0,1))$,
subject to the Dirichlet and Neumann boundary condition
at~$0$ and~$1$, respectively.
A straightforward calculation yields
$$
  Q_\eps[\psi] - \left(\frac{\pi}{2\eps}\right)^2 \|\psi\|_\eps^2
  = \int_I
  \left(
  a_\eps(s)\,|\varphi'(s)|^2 + \frac{\kappa(s)}{\eps}\,|\varphi(s)|^2
  \right)
  ds
  \,,
$$
where
$$
  a_\eps(s) := \int_0^1 \frac{|\chi_1(t)|^2}{h_\eps(s,t)} \, dt
  \,.
$$
Note that $\sup a_\eps = 1+\mathcal{O}(\eps)$
due to~\eqref{uniform} and the normalization of~$\chi_1$.
Using in addition the boundedness of~$\kappa$
and
$
  \|\varphi\|_{\sii(I)} = \|\psi\|
$
together with~\eqref{norms}, we can therefore write
$$
\displaystyle
  \frac{Q_\eps[\psi]}{\,\|\psi\|_\eps^2}
  - \left(\frac{\pi}{2\eps}\right)^2
  - \mathcal{O}(1)
  \, \leq \,
  \big[1+\mathcal{O}(\eps)\big] \,
  \frac{
  \int_I
  \left(
  |\varphi'(s)|^2 + \frac{\kappa(s)}{\eps}\,|\varphi(s)|^2
  \right)
  ds
  }
  {\int_I |\varphi(s)|^2 \, ds}
  \,.
$$
From this inequality,
the minimax principle gives the upper bound
\begin{align}\label{upper}
  \lambda_j(H_\eps)
  - \left(\frac{\pi}{2\eps}\right)^2
  & \, \leq \, \big[1+\mathcal{O}(\eps)\big] \,
  \lambda_j\big(-\Delta_D^I + \frac{\kappa}{\eps}\big)
  + \mathcal{O}(1)
  \nonumber \\
  & \, = \, \lambda_j\big(-\Delta_D^I + \frac{\kappa}{\eps}\big)
  + \mathcal{O}(1)
  \qquad\mbox{as}\qquad
  \eps \to 0 \
\end{align}
for all $j \geq 1$.
Here the equality follows by~\eqref{strong}.

\subsection{Lower bound}\label{Sec.lower}
%
For all $\psi \in D(Q_\eps)$,
we have
$$
  Q_\eps[\psi] \geq
  \int_{I\times(0,1)} \frac{|\partial_1\psi(s,t)|^2}{h_\eps(s,t)} \, ds\,dt
  + \int_{I\times(0,1)} \frac{\nu\big(\eps\kappa(s)\big)}{\eps^2}
  \ |\psi(s,t)|^2
  \, h_\eps(s,t) \, ds\,dt
  \,,
$$
where $\nu(\epsilon) \equiv \lambda_1(T_\epsilon)$
denotes the lowest eigenvalue
of the operator~$T_\epsilon$ in the Hilbert space
$\sii\big((0,1),(1-\epsilon t)dt\big)$
defined by
\begin{align*}
  (T_\epsilon\chi)(t)
  &:= -\chi''(t) + \frac{\epsilon}{1-\epsilon t} \, \chi'(t)
  \,,
  \\
  \chi \in D(T_\epsilon)
  &:= \left\{
  \chi \in W^{2,2}\big((0,1)\big) \ | \quad
  \chi(0) = \chi'(1) = 0
  \right\}
  .
\end{align*}
Note that $\nu(0)=(\pi/2)^2$ and that the corresponding
eigenfunction for $\epsilon=0$ can be identified with~$\chi_1$.
The analytic perturbation theory yields
\begin{equation}\label{analytic}
  \nu(\epsilon) = \left(\frac{\pi}{2}\right)^2 + \epsilon
  + \mathcal{O}\big(\epsilon^2\big)
  \qquad\mbox{as}\qquad
  \epsilon \to 0
  \,.
\end{equation}
Using this expansion and the boundedness of~$\kappa$,
we can estimate
$$
  Q_\eps[\psi] - \left(\frac{\pi}{2\eps}\right)^2 \|\psi\|_\eps^2
  \geq \int_{I\times(0,1)}
  \left(
  \frac{|\partial_1\psi(s,t)|^2}{1-\eps \inf\kappa}
  + \frac{\kappa}{\eps} \, |\psi(s,t)|^2
  - C \, |\psi(s,t)|^2
  \right)
  ds \, dt
  \,,
$$
where~$C$ is a positive constant depending uniquely
on $\|\kappa\|_{L^\infty(I)}$.
Using in addition~\eqref{norms}, we therefore get
$$
\displaystyle
  \frac{Q_\eps[\psi]}{\,\|\psi\|_\eps^2}
  - \left(\frac{\pi}{2\eps}\right)^2
  - \mathcal{O}(1)
  \, \geq \,
  \big[1+\mathcal{O}(\eps)\big] \,
  \frac{
  \int_{I\times(0,1)}
  \left(
  |\partial_1\psi(s,t)|^2 + \frac{\kappa(s)}{\eps}\,|\psi(s,t)|^2
  \right)
  ds
  }
  {\int_{I\times(0,1)} |\psi(s,t)|^2 \, ds}
  \,.
$$
Consequently, the minimax principle gives
\begin{align}\label{lower}
  \lambda_j(H_\eps)
  - \left(\frac{\pi}{2\eps}\right)^2
  & \, \geq \, \big[1+\mathcal{O}(\eps)\big] \,
  \lambda_j\big(-\Delta_D^I + \frac{\kappa}{\eps}\big)
  + \mathcal{O}(1)
  \nonumber \\
  & \, = \, \lambda_j\big(-\Delta_D^I + \frac{\kappa}{\eps}\big)
  + \mathcal{O}(1)
  \qquad\mbox{as}\qquad
  \eps \to 0 \
\end{align}
for all $j \geq 1$.
Again, here the equality follows by~\eqref{strong}.

In view of the unitary equivalence of~$H_\eps$
with $-\Delta_{DN}^{\Omega_\eps}$,
the estimates~\eqref{upper} and~\eqref{lower}
prove Theorem~\ref{Thm.stronger}.

\section{Norm-resolvent convergence}\label{Sec.norm}
%
In this section we study the mechanism which is behind
the eigenvalue convergence of Theorem~\ref{Thm.mine} in more details.
First we explain what we mean by the norm-resolvent convergence
of the family of operators $\{-\Delta_{DN}^{\Omega_\eps}\}_{\eps>0}$.

\subsection{The reference Hilbert space and the result}
%
In Section~\ref{Sec.coord},
we identified the Laplacian $-\Delta_{DN}^{\Omega_\eps}$
with a Laplace-Beltrami-type operator~$H_\eps$ in~$\Hilbert_\eps$.
Now it is more convenient to pass to another
unitarily equivalent operator~$\hat{H}_\eps$
which acts in the ``fixed'' (\ie~$\eps$-independent) Hilbert space
$$
  \Hilbert_0 := \sii\big(I\times(0,1)\big)
  \,.
$$
This is enabled by means of the unitary mapping
\begin{equation*}
  \hat{U}_\eps : \Hilbert_\eps \to \Hilbert_0:
  \big\{\psi\mapsto \sqrt{h_\eps} \ \psi\big\}
  \,,
\end{equation*}
provided that the curvature~$\kappa$
is differentiable in a weak sense;
henceforth we assume that
\begin{equation}\label{Ass.derivative}
  \kappa' \in L^\infty(I)
  \,.
\end{equation}
We set $\hat{H}_\eps:=\hat{U}_\eps H_\eps \hat{U}_\eps^{-1}$.
As a comparison operator to $\hat{H}_\eps$ for small~$\eps$,
we consider the decoupled operator
$$
  \hat{H}_0 :=
  \left(-\Delta_D^I + \frac{\kappa}{\eps}\right) \otimes 1
  + 1 \otimes \Big(-\frac{1}{\eps^{2}}\,\Delta_{DN}^{(0,1)}\Big)
  \qquad\mbox{in}\qquad
  \sii(I)\otimes\sii\big((0,1)\big)
  \,.
$$
Here the subscript~$0$ is just a notational convention, of course,
since~$\hat{H}_0$ still depends on~$\eps$.
Using natural isomorphisms, we may reconsider~$\hat{H}_0$
as an operator in~$\Hilbert_0$.

We clearly have
\begin{equation}\label{pre.lb1}
  \hat{H}_0 \geq \left(\frac{\pi}{2\eps}\right)^2 + \frac{\inf\kappa}{\eps}
  \,.
\end{equation}
At the same time,
\begin{equation}\label{pre.lb2}
  \hat{H}_\eps
  \geq \frac{\nu(\eps\kappa)}{\eps^2}
  \geq \frac{\nu(\eps\inf\kappa)}{\eps^2}
  = \left(\frac{\pi}{2\eps}\right)^2 + \frac{\inf\kappa}{\eps}
  + \mathcal{O}(1)
  \,,
\end{equation}
where the first inequality was established
(for the unitarily equivalent operator~$H_\eps$)
in the beginning of Section~\ref{Sec.lower},
the second inequality holds due to the monotonicity of
$\eps\mapsto\nu(\eps)$ proved in \cite[Thm.~2]{FK3}
and the equality follows from~\eqref{analytic}.
(Alternatively, we could use Theorem~\ref{Thm.mine} to get~\eqref{pre.lb2},
however, one motivation of the present section is to show
that the former can be obtained as a consequence of Theorem~\ref{Thm.norm} below.)
Fix any number
\begin{equation}\label{k}
  k > -\inf\kappa \,.
\end{equation}
It follows that $\hat{H}_\eps-\pi^2/(2\eps)^2+k/\eps$
and $\hat{H}_0-\pi^2/(2\eps)^2+k/\eps$ are positive operators
for all sufficiently small~$\eps$.

Now we are in a position to state the main result of this section.
\begin{Theorem}\label{Thm.norm}
In addition to the injectivity of~$\mathcal{L}_\eps$
and the boundedness of~$\kappa$, let us assume~\eqref{Ass.derivative}.
Then there exist positive constants~$\eps_0$ and~$C_0$,
depending uniquely on~$k$
and the supremum norms of~$\kappa$ and~$\kappa'$,
such that for all $\eps\in(0,\eps_0)$:
\begin{equation*}
  \left\|
  \left[\hat{H}_\eps-\left(\frac{\pi}{2\eps}\right)^2+\frac{k}{\eps}\right]^{-1}
  - \left[\hat{H}_0-\left(\frac{\pi}{2\eps}\right)^2+\frac{k}{\eps}\right]^{-1}
  \right\|
  \ \leq \
  C_0 \, \eps^{3/2}
  \,.
\end{equation*}
\end{Theorem}

The theorem is proved in several steps
divided into the following subsections.
In particular, it follows as a direct consequence
of Lemmata~\ref{Lem.inter} and~\ref{Lem.complement} below.
In the final subsection we show how it implies
the convergence of eigenvalues of Theorem~\ref{Thm.mine}.

\subsection{The transformed Laplacian}
%
Let us now find an explicit expression for the quadratic form~$\hat{Q}_\eps$
associated with the operator~$\hat{H}_\eps$.
By definition, it is given by
\begin{equation*}
  \hat{Q}_\eps[\psi] := Q_\eps[\hat{U}_\eps^{-1}\psi] \,,
  \qquad
  \psi \in D(\hat{Q}_\eps) := \hat{U}_\eps D(Q_\eps)
  \,.
\end{equation*}
One easily verifies that
\begin{equation}\label{form.domain}
  D(\hat{Q}_\eps) = D(Q_\eps) =: \mathcal{Q}
  \,,
\end{equation}
which is actually independent of~$\eps$.
Furthermore, for any $\psi \in \mathcal{Q}$,
we have
$$
  \hat{Q}_\eps[\psi] = \hat{Q}_\eps^1[\psi]+\hat{Q}_\eps^2[\psi]
  \,,
$$
where
\begin{align*}
  \hat{Q}_\eps^1[\psi]
  &:=
  \int
  \frac{\big|\partial_1 (h_\eps^{-1/2}\psi)\big|^2}{h_\eps}
  &=&
  \int \left\{
  \frac{|\partial_1\psi|^2}{h_\eps^2}
  + V_\eps^1 \, |\psi|^2
  + V_\eps^2 \, \Re\big(\overline{\psi}\partial_1\psi\big)
  \right\} ,
  \\
  \hat{Q}_\eps^2[\psi]
  &:=
  \int \frac{\big|\partial_2(h_\eps^{-1/2}\psi)\big|^2}{\eps^2} \ h_\eps
  &=& \int \left\{
  \frac{|\partial_2\psi|^2}{\eps^2}
  + V_\eps^3 \, |\psi|^2
  + V_\eps^4  \, \Re\big(\overline{\psi}\partial_2\psi\big)
  \right\} ,
\end{align*}
with
\begin{align*}
 V_\eps^1(s,t) &:=  \frac{1}{4} \frac{\kappa'(s)^2 \eps^2 t^2}{h_\eps(s,t)^4} \,,
 &
 V_\eps^2(s,t) &:= \frac{\kappa'(s) \eps t}{h_\eps(s,t)^3} \,,
 \\
 V_\eps^3(s,t) &:= \frac{1}{4} \frac{\kappa(s)^2}{h_\eps(s,t)^2} \,,
 &
 V_\eps^4(s,t) &:= \frac{\kappa(s)}{\eps h_\eps(s,t)} \,.
\end{align*}
Here and in the sequel the integral sign~$\int$ refers
to an integration over $I\times(0,1)$.
Integrating by parts in the expression for~$\hat{Q}_\eps^2[\psi]$,
we finally arrive at
$$
  \hat{Q}_\eps[\psi] =
  \int \left\{
  \frac{|\partial_1\psi|^2}{h_\eps^2}
  + \frac{|\partial_2\psi|^2}{\eps^2}
  + (V_\eps^1-V_\eps^3)  |\psi|^2
  + V_\eps^2 \, \Re\big(\overline{\psi}\partial_1\psi\big)
  \right\}
  + \int_\partial v_\eps \, |\psi|^2
  \,,
$$
where
$$
  v_\eps(s,t) := \frac{1}{2} \frac{\kappa(s)}{\eps \big(1-\eps\kappa(s)\big)}
  \,.
$$
Here and in the sequel the integral sign~$\int_\partial$ refers
to an integration over the boundary $I\times\{1\}$.
\begin{Remark}\label{Rem.domain}
$\hat{H}_\eps$ is exactly
the operator $H_\eps^{DN}$ mentioned briefly in Introduction.
Let us remark in this context that,
contrary to the form domains~\eqref{form.domain},
the operator domains of~$H_\eps$ and~$\hat{H}_\eps$ do differ
(unless the curvature~$\kappa$ vanishes identically).
Indeed, under additional regularity conditions about~$\gamma$,
it can be shown that while functions from $D(H_\eps)$
satisfy Neumann boundary conditions on $I\times\{1\}$,
the functions~$\psi$ from $D(\hat{H}_\eps)$ satisfy
non-homogeneous Robin-type boundary conditions
$
  \partial_2\psi + \eps^2 v_\eps \psi = 0
$
on $I\times\{1\}$.
This is the reason why the decoupling of~$\hat{H}_\eps$
for small~$\eps$ is not obvious in this situation.
At the same time, we see that the operator domain of~$\hat{H}_\eps$
heavily depends on the geometry of~$\gamma$.
For our purposes, however, it will be enough to work
with the associated quadratic form~$\hat{Q}_\eps$
whose domain is independent of~$\eps$ and~$\kappa$.
\end{Remark}
%

\subsection{Renormalized operators and resolvent bounds}
%
It will be more convenient to work with the shifted operators
$$
  L_\eps := \hat{H}_\eps-\left(\frac{\pi}{2\eps}\right)^2+\frac{k}{\eps}
  \qquad\mbox{and}\qquad
  L_0 := \hat{H}_0-\left(\frac{\pi}{2\eps}\right)^2+\frac{k}{\eps}
  \,.
$$
Let~$l_\eps$ and~$l_0$ denote the associated quadratic forms.
It is important that they have the same domain~$\mathcal{Q}$.
More precisely, $\hat{H}_0$~was initially defined as a direct sum,
however, using natural isomorphisms, it is clear that
we can identify the form domain of~$L_0$ with~$\mathcal{Q}$ and
$$
  l_0[\psi] =
  \int \left\{
  |\partial_1\psi|^2
  + \frac{1}{\eps^2}
  \Big[
  |\partial_2\psi|^2-\left(\frac{\pi}{2}\right)^2|\psi|^2
  \Big]
  + \frac{k+\kappa}{\eps} \, |\psi|^2
  \right\}
$$
for all $\psi \in \mathcal{Q}$.

It will be also useful to have an intermediate operator~$L$,
obtained from~$L_\eps$ after neglecting its non-singular dependence on~$\eps$
but keeping the boundary term.
For simplicity, henceforth we assume that~$\eps$ is less than one
and that it is in fact so small that~\eqref{Ass.basic}
holds with a number less than one on the right hand side.
Consequently,
\begin{equation}\label{V-estimates}
  |h_\eps-1| \leq C \eps
  \,, \
  |V_\eps^1| \leq C \eps^2
  \,, \
  |V_\eps^2| \leq C \eps
  \,, \
  |V_\eps^3| \leq C
  \,, \
  |V_\eps^4| \leq C \eps^{-1}
  \,, \
  |v_\eps| \leq C \eps^{-1}
  \,.
\end{equation}
Here and in the sequel,
we use the convention that~$C$ and~$c$ are positive constants
which possibly depend on~$k$ and the supremum norms of~$\kappa$ and~$\kappa'$,
and which may vary from line to line.
In view of these estimates, it is reasonable to introduce~$L$
as the operator associated with the quadratic form~$l$
defined by $D(l):=\mathcal{Q}$ and
$$
  l[\psi] :=
  \int \left\{
  |\partial_1\psi|^2
  + \frac{1}{\eps^2}
  \Big[
  |\partial_2\psi|^2-\left(\frac{\pi}{2}\right)^2|\psi|^2
  \Big]
  + \frac{k}{\eps} \, |\psi|^2
  \right\}
  + \int_\partial v_\eps \, |\psi|^2
$$
for all $\psi \in \mathcal{Q}$.
Indeed, it follows from~\eqref{V-estimates} that
\begin{align}\label{diff.inter}
  \big|l_\eps[\psi] - l[\psi]\big|
  &\leq \int \Big\{
  |h_\eps^{-2}-1| |\partial_1\psi|^2
  + |V_\eps^1-V_\eps^3| |\psi|^2
  + |V_\eps^2| \, |\psi||\partial_1\psi|
  \Big\}
  \nonumber \\
  &\leq C \big(
  \eps \|\partial_1\psi\|^2 + \|\psi\|^2
  \big)
\end{align}
for all $\psi\in\mathcal{Q}$.

Let us now argue that, for every $\psi \in \mathcal{Q}$
and for all sufficiently small~$\eps$
(which precisely means that~$\eps$ has to be less than an explicit constant
depending on~$k$ and the supremum norms of~$\kappa$ and~$\kappa'$),
we have
\begin{equation}\label{lbs}
  \min\left\{
  l_\eps[\psi], l_0[\psi], l[\psi]
  \right\}
  \geq
  c \,
  \left(
  \|\partial_1\psi\|^2 + \eps^{-1} \|\psi\|^2
  \right)
  \,.
\end{equation}
Here the bound for~$l_0$ follows at once by improving
the crude bound~\eqref{pre.lb1} and recalling~\eqref{k}.
The bound for~$l$ follows from that for~$l_\eps$ and from~\eqref{diff.inter}.
As for the bound for~$l_\eps$,
we first remark that the estimates~\eqref{pre.lb2}
actually hold for the part of~$\hat{H}_\eps$
associated with~$\hat{Q}_\eps^2$.
Second, using~\eqref{V-estimates} and some elementary estimates,
we have
$
  \hat{Q}_\eps^1[\psi]
  \geq (c-C\eps) \|\partial_1\psi\|^2
  - C \|\psi\|^2
$.
Hence, for~$\eps$ small enough,
we indeed conclude with the bound for~$l_\eps$.

The estimates~\eqref{lbs} imply that,
for all sufficiently small~$\eps$,
\begin{equation}\label{res.bounds}
  \|L_\eps^{-1}\| \leq C \eps
  \,, \qquad
  \|L_0^{-1}\| \leq C \eps
  \,, \qquad
  \|L^{-1}\| \leq C \eps
  \,.
\end{equation}
%

\subsection{An intermediate convergence result}
%
As the first step in the proof of Theorem~\ref{Thm.norm},
we show that it is actually enough to establish
the norm-resolvent convergence for a simpler operator~$L$ instead of~$L_\eps$.
\begin{Lemma}\label{Lem.inter}
Under the assumptions of Theorem~\ref{Thm.norm},
there exist positive constants~$\eps_0$ and~$C_0$,
depending uniquely on~$k$
and the supremum norms of~$\kappa$ and~$\kappa'$,
such that for all $\eps\in(0,\eps_0)$:
\begin{equation*}
  \left\|
  L_\eps^{-1} - L^{-1}
  \right\|
  \ \leq \
  C_0 \, \eps^{2}
  \,.
\end{equation*}
\end{Lemma}
\begin{proof}
We are inspired by~\cite[Sec.~3]{Friedlander-Solomyak_2007a}.
Adapting the estimate~\eqref{diff.inter}
for the sesquilinear form generated by $l_\eps-l$
and using~\eqref{lbs}, we get
\begin{align*}
  \big|l_\eps(\phi,\psi) - l(\phi,\psi)\big|
  &\leq C
  \sqrt{\eps \|\partial_1\phi\|^2 + \|\phi\|^2}
  \sqrt{\eps \|\partial_1\psi\|^2 + \|\psi\|^2}
  \\
  &\leq (C/c) \, \eps \sqrt{l[\phi]\,l_\eps[\psi]}
\end{align*}
for every $\phi,\psi \in \mathcal{Q}$.
Choosing $\phi:=L^{-1} f$ and $\psi:=L_\eps^{-1} g$,
where $f,g\in\Hilbert_0$ are arbitrary,
we arrive at
$$
  \big|(f,L^{-1}g) - (f,L_\eps^{-1}g)\big|
  \leq (C/c) \, \eps \sqrt{(f,L^{-1}f)(g,L_\eps^{-1}g)}
  \leq (C^2/c) \, \eps^2 \, \|f\| \, \|g\|
  \,.
$$
Here $(\cdot,\cdot)$ denotes the inner product in~$\Hilbert_0$
and the second inequality follows from~\eqref{res.bounds}.
This completes the proof with $C_0:=C^2/c$.
\end{proof}
%

\subsection{An orthogonal decomposition of the Hilbert space}
%
Contrary to Lemma~\ref{Lem.inter},
the convergence of $\|L^{-1}-L_0^{-1}\|$ is less obvious.
We follow the idea of~\cite{Friedlander-Solomyak_2007a}
and decompose the Hilbert space~$\Hilbert_0$
into an orthogonal sum
$$
  \Hilbert_0 = \mathfrak{H}_1 \oplus \mathfrak{H}_1^\bot
  \,,
$$
where the subspace~$\mathfrak{H}_1$ consists of functions~$\psi_1$
such that
\begin{equation}\label{psi1}
  \psi_1(s,t) = \varphi_1(s) \chi_1(t)
  \,.
\end{equation}
Recall that~$\chi_1$ has been introduced in~\eqref{chi}.
Since~$\chi_1$ is normalized, we clearly have
$
  \|\psi_1\|=\|\varphi_1\|_{\sii(I)}
$.
Given any $\psi\in\Hilbert_0$, we have the decomposition
\begin{equation}\label{psi.decomposition}
  \psi = \psi_1 + \psi_\bot
  \qquad\mbox{with}\qquad
  \psi_1 \in \mathfrak{H}_1, \ \psi_\bot\in \mathfrak{H}_1^\bot
  \,,
\end{equation}
where~$\psi_1$ has the form~\eqref{psi1}
with $\varphi_1(s):=\int_0^1 \psi(s,t) \chi_1(t) dt$.
Note that $\psi_1\in\mathcal{Q}$ if $\psi\in\mathcal{Q}$.
The inclusion $\psi_\bot\in\mathfrak{H}_1^\bot$ means that
\begin{equation}\label{orth.identity1}
  \int_0^1 \psi_\bot(s,t) \, \chi_1(t) \, dt = 0
  \qquad\mbox{for a.e.}\quad s \in I
  \,.
\end{equation}
If in addition $\psi_\bot \in \mathcal{Q}$,
then one can differentiate the last identity to get
\begin{equation}\label{orth.identity2}
  \int_0^1 \partial_1\psi_\bot(s,t) \, \chi_1(t) \, dt = 0
  \qquad\mbox{for a.e.}\quad s \in I
  \,.
\end{equation}
%

\subsection{A complementary convergence result}
%
Now we are in a position to prove the following result,
which together with Lemma~\ref{Lem.inter}
establishes Theorem~\ref{Thm.norm}.
\begin{Lemma}\label{Lem.complement}
Under the assumptions of Theorem~\ref{Thm.norm},
there exist positive constants~$\eps_0$ and~$C_0$,
depending uniquely on~$k$
and the supremum norms of~$\kappa$ and~$\kappa'$,
such that for all $\eps\in(0,\eps_0)$:
\begin{equation*}
  \left\|
  L^{-1} - L_0^{-1}
  \right\|
  \ \leq \
  C_0 \, \eps^{3/2}
  \,.
\end{equation*}
\end{Lemma}
\begin{proof}
Again, we use some of the ideas of~\cite[Sec.~3]{Friedlander-Solomyak_2007a}.
As a consequence of~\eqref{orth.identity1} and~\eqref{orth.identity2},
we get that $l_0(\psi_1,\psi_\bot)=0$; therefore
\begin{equation}\label{nomixed}
  l_0[\psi] = l_0[\psi_1]+l_0[\psi_\bot]
\end{equation}
for every $\psi \in \mathcal{Q}$.
At the same time, for sufficiently small~$\eps$,
\begin{equation}\label{lbs.bis}
\begin{aligned}
  l_0[\psi_1]
  &\geq c
  \left(
  \|\varphi_1'\|_{\sii(I)}^2 + \eps^{-1} \, \|\varphi_1\|_{\sii(I)}^2
  \right)
  \,,
  \\
  l_0[\psi_\bot]
  &\geq c
  \Big(
  \|\partial_1\psi_\bot\|^2 + \eps^{-2} \, \|\partial_2\psi_\bot\|^2
  + \eps^{-2} \, \|\partial\psi_\bot\|^2
  \Big)
  \,,
\end{aligned}
\end{equation}
where the second inequality is based on
$\int |\partial_2\psi_\bot|^2 \geq \pi^2 \int |\psi_\bot|^2$.

Let us now compare~$l_0$ with~$l$.
For every $\psi \in \mathcal{Q}$, we define
$$
  m[\psi] := l[\psi]-l_0[\psi]
  = \int_\partial v_\eps \, |\psi|^2
  - \int \frac{\kappa}{\eps} \, |\psi|^2
  \,.
$$
Using the estimates~\eqref{V-estimates} and~\eqref{lbs.bis},
we get, for any $\psi \in \mathcal{Q}$
decomposed as in~\eqref{psi.decomposition},
\begin{align*}
  m[\psi_1]
  &=
  \int_I \frac{\kappa(s)^2}{1-\eps \kappa(s)} \, |\varphi_1(s)|^2 ds
  \leq C \, \|\varphi_1\|_{\sii(\Real)}^2
  \leq (C/c) \, \eps \, l_0[\psi_1]
  \,,
  \\
  |m[\psi_\bot]|
  &\leq C \eps^{-1} \Big(
  \|\psi_\bot\| \|\partial_2\psi_\bot|| + \|\psi_\bot\|^2
  \Big)
  \leq 2 (C/c) \, \eps \, l_0[\psi_\bot]
  \,,
  \\
  |m(\psi_1,\psi_\bot)|
  &= \left| \int_\partial
  v_\eps \, \overline{\psi_1} \psi_\bot
  \right|
  \leq C \eps^{-1} \|\psi_1\| \sqrt{\|\psi_\bot\| \|\partial_2\psi_\bot||}
  \\
  &\leq (C/c) \, \eps^{1/2} \sqrt{l_0[\psi_1]l_0[\psi_\bot]}
  \,.
\end{align*}
Except for~$m[\psi_1]$,
here the boundary integral was estimated via
$$
  \int_\partial |\psi|^2
  = \int \partial_2|\psi|^2
  = \int 2\,\Re\big(\overline{\psi}\partial_2\psi\big)
  \leq 2 \, \|\psi\| \, \|\partial_2\psi\|
  \,.
$$
Taking~\eqref{nomixed} into account,
we conclude with the estimate
(which can be again adapted for the corresponding sesquilinear form)
$$
  |m[\psi]| \leq C \, \eps^{1/2} \ l_0[\psi]
$$
valid for every $\psi \in \mathcal{Q}$
and all sufficiently small~$\eps$.
In particular, this implies the crude estimates
$$
  c \, l_0[\psi] \leq l[\psi] \leq C \, l_0[\psi]
  \,.
$$

Summing up, we have the crucial bound
$$
  \big| l(\phi,\psi)-l_0(\phi,\psi) \big|
  \leq C \, \eps^{1/2} \, \sqrt{l_0[\phi] \, l[\psi]}
  \,,
$$
valid for arbitrary $\phi,\psi \in \mathcal{Q}$.
The rest of the proof then follows the lines
of the proof of Lemma~\ref{Lem.inter}.
\end{proof}
%

\subsection{Convergence of eigenvalues}
%
As an application of Theorem~\ref{Thm.norm},
we shall show now how it implies the eigenvalue asymptotics
of Theorem~\ref{Thm.mine}.
Recall that the numbers $\lambda_j(H)$ represent
either eigenvalues below the essential spectrum
or the threshold of the essential spectrum of~$H$.
In particular, they provide a complete information
about the spectrum of~$H$ if it is an operator with compact resolvent.
In our situation, this will be the case if~$I$ is bounded,
but let us stress that we allow infinite or semi-infinite intervals, too.

We begin with analysing the spectrum of the comparison operator.
\begin{Lemma}\label{Lem.comparison}
Let~$\kappa$ be bounded.
One has
$$
  \lambda_1(\hat{H}_0)
  = \lambda_1\big(-\Delta_D^I+\frac{\kappa}{\eps}\big)
  + \left(\frac{\pi}{2\eps}\right)^2
  \,.
$$
Moreover, for any integer $N \geq 2$,
there exists a positive constant~$\eps_0$
depending on~$N$, $\kappa$ and~$I$
such that for all $\eps < \eps_0$:
$$
  \forall j \in\{1,\dots,N\}, \qquad
  \lambda_j(\hat{H}_0)
  = \lambda_j\big(-\Delta_D^I+\frac{\kappa}{\eps}\big)
  + \left(\frac{\pi}{2\eps}\right)^2
  \,.
$$
\end{Lemma}
\begin{proof}
Since~$\hat{H}_0$ is decoupled,
we know that (\cf~\cite[Corol.\ of Thm.~VIII.33]{RS1})
\begin{equation*}
  \left\{\lambda_j(\hat{H}_0)\right\}_{j=1}^\infty
  = \left\{\lambda_j\big(-\Delta_D^I+\frac{\kappa}{\eps}\big)\right\}_{j=1}^\infty
  + \left\{ \left(\frac{j\pi}{2\eps}\right)^2 \right\}_{j=1}^\infty
  \,,
\end{equation*}
and it only remains to arrange the sum of the numbers
on the right hand side into a non-decreasing sequence.
The assertion for $N=1$ is therefore trivial.
Let $j \geq 2$ and assume by induction that
$
  \lambda_{j-1}(\hat{H}_0) - \pi^2/(2\eps)^2
  = \lambda_{j-1}(-\Delta_D^I+\kappa/\eps)
$.
Then
$$
  \lambda_{j}(\hat{H}_0) - \pi^2/(2\eps)^2
  = \min\left\{
  \lambda_{j-1}(-\Delta_D^I+\kappa/\eps),
  3\pi^2/(2\eps)^2
  \right\}
$$
and the assertion of Lemma follows at once
due to the asymptotics~\eqref{strong}.
\end{proof}

Now, fix $j \geq 1$ and assume that~$\eps$ is so small
that the conclusions of Theorem~\ref{Thm.norm}
and Lemma~\ref{Lem.comparison} hold.
By virtue of Theorem~\ref{Thm.norm}, we have
$$
  \left|
  \left[\lambda_j(\hat{H}_\eps)
  -\left(\frac{\pi}{2\eps}\right)^2+\frac{k}{\eps}\right]^{-1}
  - \left[\lambda_j(\hat{H}_0)
  -\left(\frac{\pi}{2\eps}\right)^2+\frac{k}{\eps}\right]^{-1}
  \right|
  \ \leq \
  C_0 \, \eps^{3/2}
  \,,
$$
since the left hand side is estimated
by the norm of the resolvent difference.
Using now Lemma~\ref{Lem.comparison},
the above estimate is equivalent to
\begin{equation}\label{1/2}
  \left|
  \frac{1}{\eps \big[\lambda_j(\hat{H}_\eps)-\pi^2/(2\eps)^2\big]+k}
  -\frac{1}{\eps \, \lambda_j(-\Delta_D^I+\kappa/\eps)+k}
  \right|
  \ \leq \
  C_0 \, \eps^{1/2}
  \,.
\end{equation}
Consequently, recalling~\eqref{strong}, we conclude with
$$
  \lim_{\eps \to 0} \eps \left[\lambda_j(\hat{H}_\eps)
  -\left(\frac{\pi}{2\eps}\right)^2\right]
  = \lim_{\eps \to 0} \eps \, \lambda_j\big(-\Delta_D^I+\frac{\kappa}{\eps}\big)
  = \inf\kappa
  \,.
$$
This is indeed equivalent to Theorem~\ref{Thm.mine}
because~$\hat{H}_\eps$ and~$-\Delta_{DN}^{\Omega_\eps}$
are unitarily equivalent (therefore isospectral).

\begin{Remark}
Because of the lack of one half in the power of~$\eps$
in Theorem~\ref{Thm.norm}, it turns out that~\eqref{1/2}
yields a slightly worse result than Theorem~\ref{Thm.stronger}.
Let us also emphasize that Theorem~\ref{Thm.mine} and~\ref{Thm.stronger}
have been proved in Section~\ref{Sec.proof}
without the need to assume the extra condition~\eqref{Ass.derivative}.
On the other hand, Theorem~\ref{Thm.norm} contains an operator-type
convergence result of independent interest.
\end{Remark}
%

\section{Possible extensions}\label{Sec.end}
%
\subsection{Different boundary conditions}
The result of Theorem~\ref{Thm.stronger}
readily extends to the case of other boundary conditions
imposed on the sides
$\mathcal{L}_\eps\big(\{\inf I\}\times(0,1)\big)$
and $\mathcal{L}_\eps\big(\{\sup I\}\times(0,1)\big)$,
provided that the boundary conditions for the one-dimensional
Schr\"odinger operator are changed accordingly.

It is more interesting to impose different boundary
conditions on the approaching parallel curves as $\eps \to 0$.
As an example, let us keep the Dirichlet boundary conditions
but replace the Neumann boundary condition
by the Robin condition of the type
$$
  \frac{\partial \Psi}{\partial n}
  + (\alpha \circ \ell_\eps^{-1}) \, \Psi = 0
  \qquad\mbox{on}\qquad
  \gamma_\eps(I)
  \,.
$$
Here $\ell_\eps:=\mathcal{L}_\eps(\cdot,1)$
and~$\alpha:I\to\Real$
is assumed to be bounded and uniformly continuous.
Let us denote the corresponding Laplacian by
$-\Delta_{DR_\alpha}^{\Omega_\eps}$.
Then the method of the present paper gives
\begin{Theorem}\label{Thm.Robin}
For all $j \geq 1$,
\begin{align*}
  \lambda_j(-\Delta_{DR_\alpha}^{\Omega_\eps})
  &= \left(\frac{\pi}{2\eps}\right)^2
  + \lambda_j\big(-\Delta_D^I + \frac{\kappa+2\alpha}{\eps}\big)
  + \mathcal{O}(1)
  \\
  &= \left(\frac{\pi}{2\eps}\right)^2
  + \frac{\inf(\kappa+2\alpha)}{\eps}
  + o(\eps^{-1})
  \qquad\mbox{as}\qquad
  \eps \to 0
  \,.
\end{align*}
\end{Theorem}
Let us mention that strips with this combination
of Dirichlet and Robin boundary conditions
were studied in~\cite{FK3}.

\subsection{Higher-dimensional generalization}
Let~$\Omega_\eps$ be a three-dimensional layer instead of the planar strip.
That is, we keep the definition~\eqref{strip} with~\eqref{StripMap},
but now $\gamma: I \subseteq \Real^2 \to \Real^3$ is a parametrization of
a two-dimensional surface embedded in~$\Real^3$
and $n:=(\partial_1\gamma)\times(\partial_2\gamma)$,
where the cross denotes the vector product in $\Real^3$.
Let~$M$ be the corresponding mean curvature.
Proceeding in the same way as in Section~\ref{Sec.proof}
(we omit the details but refer to~\cite{DEK2}
for a necessary geometric background), we get
\begin{Theorem}\label{Thm.layer}
For all $j \geq 1$,
\begin{align*}
  \lambda_j(-\Delta_{DN}^{\Omega_\eps})
  &= \left(\frac{\pi}{2\eps}\right)^2
  + \lambda_j\big(-\Delta_D^\gamma + \frac{2M}{\eps}\big)
  + \mathcal{O}(1)
  \\
  &= \left(\frac{\pi}{2\eps}\right)^2
  + \frac{2\inf M}{\eps}
  + o(\eps^{-1})
  \qquad\mbox{as}\qquad
  \eps \to 0
  \,,
\end{align*}
where $-\Delta_D^\gamma$ denotes the Laplace-Beltrami operator
in $\sii\big(\gamma(I)\big)$, subject to Di\-richlet boundary conditions.
\end{Theorem}

Notice that the leading geometric term in the asymptotic expansions
depends on the extrinsic curvature only.
This suggests that the spectral properties of the Dirichlet-Neumann layers
will differ significantly from the purely Dirichlet case
studied in~\cite{DEK2,CEK,LL1,LL3},
where the Gauss curvature of~$\gamma$ plays a crucial role.

\subsection{Curved ambient space}
The results of the present paper extend to the case
of strips embedded in an abstract two-dimensional Riemannian manifold~$\mathcal{A}$
instead of the Euclidean plane.
Indeed, it follows from~\cite{K1} (see also~\cite{K3} and~\cite[Sec.~5]{FK4})
that the quadratic form~$Q_\eps$ has the same structure;
the only difference is that in this more general situation
$h_\eps$~is obtained as the solution of the Jacobi equation
\begin{equation*}
  \partial_2^2 h_\eps + \eps^2\,K\,h_\eps = 0
  \qquad\textrm{with}\qquad\left\{
  \begin{aligned}
    h_\eps(\cdot,0) &= 1 \,, \\
    \partial_2 h_\eps(\cdot,0) &= -\eps\,\kappa \,,
  \end{aligned}
  \right.
\end{equation*}
where~$K$ is the Gauss curvature of~$\mathcal{A}$.
Here~$\kappa$ is the curvature of $\gamma:I\to\mathcal{A}$
(it is in fact the geodesic curvature of~$\gamma$
if the ambient space~$\mathcal{A}$ is embedded in~$\Real^3$).
Consequently, up to higher-order terms in~$\eps$,
the function $h_\eps$ coincides with the expression~\eqref{Jacobian}
for the flat case $K=0$. Namely,
$$
  h_\eps(s,t) = 1 - \kappa(s) \, \eps \, t + \mathcal{O}(\eps^2)
  \qquad\mbox{as}\qquad
  \eps \to 0
  \,.
$$
Following the lines of the proof in Section~\ref{Sec.proof},
it is then possible to check that Theorems~\ref{Thm.mine}
and~\ref{Thm.stronger} remain valid without changes.
The curvature of the ambient space comes into the asymptotics
via higher-order terms only.

\subsection*{Acknowledgment}
The work has been supported by 
the Czech Academy of Sciences and its Grant Agency
within the projects IRP AV0Z10480505 and A100480501,
and by the project LC06002 of the Ministry of Education,
Youth and Sports of the Czech Republic.

%
%
\providecommand{\bysame}{\leavevmode\hbox to3em{\hrulefill}\thinspace}
\providecommand{\MR}{\relax\ifhmode\unskip\space\fi MR }
\providecommand{\MRhref}[2]{%
  \href{http://www.ams.org/mathscinet-getitem?mr=#1}{#2}
}
\providecommand{\href}[2]{#2}

\end{document}